\documentclass{proc-l}

\newtheorem{theorem}{Theorem}[section]

\newtheorem{proposition}[theorem]{Proposition}
\newtheorem{corollary}[theorem]{Corollary}

\theoremstyle{definition}

\theoremstyle{remark}
\newtheorem{remark}[theorem]{Remark}

\numberwithin{equation}{section}

\renewcommand{\P}{\mathbb{P}}

\newcommand{\C}{\mathbb{C}}

\newcommand{\fk}{\mathfrak}
\newcommand{\fkS}{\mathfrak{S}}

\newcommand{\RC}{\mathrm{RC}}
\newcommand{\Red}{\mathrm{Red}}
\newcommand{\BB}{\mathrm{BB}}

\renewcommand{\fk}[1]{\mathfrak{#1}}

\begin{document}

\title[]{dual Schubert polynomials via a cauchy identity}

\date{\today}

\author{Zachary Hamaker}

\begin{abstract}
	We give a combinatorial proof that Postnikov and Stanley's formula for dual Schubert polynomials in terms of weighted chains in Bruhat order is equivalent to a classical Cauchy identity for polynomials.
	This gives a natural interpretation of Huang and Pylyavskyy's recent insertion algorithms.
\end{abstract}

\maketitle

\section{Introduction}

Recall $S_\infty$ is the direct limit of the sequence of symmetric groups under the natural inclusion $S_n \hookrightarrow S_{n+1}$.
The Schubert polynomials $\{\fk S_w\}_{w \in S_\infty}$ were introduced in~\cite{lascoux1982structure} by Lascoux and Sch\"utzenberger as specific representatives for the cohomology of the flag variety.
They are defined using the divided difference operators of Bernstein, Gelfand and Gelfand~\cite{bernstein1973schubert}.
Schubert polynomials are multivariate polynomials indexed by permutations and form a basis for the polynomial ring $\mathbb{C}[x_1,x_2,\dots]$.
Therefore, Schubert products expand into Schubert polynomials:
\[
\fkS_u \cdot \fkS_v = \sum_w c^w_{uv} \fkS_w
\]
The quantities $c^w_{uv}$, which we call \emph{(generalized) Littlewood--Richardson coefficients}, are known to be non-negative integers since they are structure coefficients for the cup product in the cohomology of the complete flag variety.
Giving a combinatorial formula for Littlewood--Richardson coefficients is widely acknowledged as one of the most important open problems in algebraic combinatorics.
Only special cases are known, beginning with Monk's rule for computing $\fkS_{s_k} \cdot \fkS_w$~\cite{monk1959geometry}.

For $f,g \in \mathbb{C}[x_1,x_2,\dots]$, define the inner product
\begin{equation}
	\label{eq:inner_product}
	\langle f, g \rangle := f\left(\frac{\partial}{\partial x_1},\frac{\partial}{\partial x_2}, \dots\right) \cdot g(x_1,x_2, \dots)  
	\mid_{x_1 = x_2 = \dots = 0}.
\end{equation}
Note this inner product restricts to any finite set of variables.
Dual bases for $\langle \cdot, \cdot \rangle$ are characterized by a classical Cauchy identity.
Let $P_n: \C[x_1,x_2,\dots] \to \C[x_1,\dots,x_n]$ be the obvious projection operator.
\begin{theorem}[Cauchy identity, {see~\cite[Thm~5.2]{postnikov2009chains}}]
\label{t:cauchy}
Let $\{f_w\},\{g_w\} \subseteq \C[x_1,x_2,\dots]$ be homogeneous  so that $\{P_n(f_w)\} \setminus \{0\}$ and $\{P_n(g_w)\}\setminus \{0\}$ are bases of $\C[x_1,\dots,x_n]$ for any $n$.\footnote{We need this projection as the `dual basis' we will consider fails to be a basis of $\C[x_1,x_2,\dots]$.
Instead, it is a basis of $\C[x_1-x_2,x_2-x_3,\dots]$ and restricts to bases of $\C[x_1,\dots,x_n]$ for any $n$.}
Then $\langle f_u, g_v \rangle = \delta_{uv}$ if and only if
\begin{equation}
\label{eq:cauchy}
e^{x_1y_1 + x_2y_2 + \dots} = \sum_{w \in S_\infty} f_w(x_1,x_2,\dots) \cdot g_w(y_1,y_2,\dots).	
\end{equation}	
\end{theorem}
In a slightly different guise, Bernstein, Gelfand and Gelfand~\cite{bernstein1973schubert} introduced \emph{dual Schubert polynomials} $\{\fk D_w\}_{w \in S_\infty}$, which are dual to Schubert polynomials in that
\begin{equation}
	\langle \fkS_u, \fk D_w \rangle = \delta_{uw}.
\label{eq:SD}
\end{equation}
Postnikov and Stanley made this connection explicit and give a detailed formal treatment of dual Schubert polynomials, developing their basic properties~\cite{postnikov2009chains}.
Write $u \to v$ when $u$ is covered by $w$ in Bruhat order and $\mu(u \to v) = x_a - x_b$ where $v = u(a,b)$.
They define the \emph{skew dual Schubert polynomial} for $u,w \in S_\infty$ as
\begin{equation}
\label{eq:dual-schub}
	\fk{D}_{u,w} := \frac{1}{p!} \sum_{C: u = v_0 \to v_1 \to v_2 \to \dots \to v_p = w} \prod_{i = 1}^p \mu(v_{i-1} \to v_i).
\end{equation}
Using methods from linear algebra and algebraic geometry, they  prove:
\begin{theorem}[{\cite{postnikov2009chains}}]
\label{t:dual-def}
	For $w \in S_\infty$, $\fk{D}_w = \fk{D}_{1,w}$.
\end{theorem}

We give a new bijective proof that $\fk{D}_w = \fk{D}_{1,w}$ using the Cauchy identity (Theorem~\ref{t:cauchy}) and Monk's rule.
The first step is a straightforward reinterpretation of the Postnikov-Stanley formula in terms of labeled chains in Bruhat order.
After interpreting the left side of Equation~\eqref{eq:cauchy} appropriately, we show every bijective proof of Monk's formula gives a bijective proof that Equations~\eqref{eq:dual-schub} and~\eqref{eq:SD} are equivalent.
In particular, we interpret the insertions introduced by Huang and Pylyavskyy in~\cite{huang2022bumpless} as bijective proofs of this Cauchy identity.
As a consequence, we also uncover surprising symmetries and redundancies in the formulas for dual Schubert polynomials.

\section{background}

Let $\mathbb{P} = \{1,2 \dots\}$ be the set of positive integers.

\subsection{Permutations} 
Let $S_\infty$ be the set of permutations of $\mathbb{P}$ fixing all but finitely many elements.
This is a group generated by the simple transpositions $s_i = (i,i+1)$ for $i > 0$.
For $w \in S_\infty$, the \emph{length} $\ell(w)$ is the minimum number of generators in an expression $w = s_{a_1} \dots s_{a_p}$.
Such expressions are called \emph{reduced}, and $\Red(w)$ denotes the set of reduced expressions for $w$.

The \emph{Bruhat order} $\leq$ on $S_\infty$ is the partial order with cover relations $u \xrightarrow{} v$ if $v = u (a,b)$ with $\ell(v) = \ell(u) + 1$.
We say $v$ is a \emph{$k$--Bruhat cover} of $u$, denoted $u \xrightarrow{k} v$, if $a \leq k < b$.
For $u,w \in S_\infty$ with $u \leq w$, a \emph{chain} from $u$ to $w$ is a sequence $C$ of cover relations from $u$ to $w$
\[
C : u = v_0 \xrightarrow{} v_1 \xrightarrow{} \dots \xrightarrow{} v_p = w.
\]
Similarly, a \emph{labeled chain} is a chain of $k$--Bruhat covers
\[
C : u = v_0 \xrightarrow{k_1} v_1 \xrightarrow{k_2} \dots \xrightarrow{k_p} v_p = w.
\]
For $\textbf{k} = (k_1, \dots, k_p)$, let $C: u \xrightarrow{\textbf{k}} w$ denote a labeled chain with these labels.
Let $\mathcal{C}(u,v)$ denote the set of chains from $u$ to $v$.
Similarly, let $\mathcal{C}_\textbf{k}(u,v)$ be the set of labeled chains from $u$ to $v$ whose labels are $\textbf{k}$.

\subsection{Schubert polynomials} The material in this subsection appears in~\cite{manivel2001symmetric} unless another reference is given.
Schubert polynomials were originally defined algebraically, but now have several combinatorial definitions.
We will work with the Billey--Jockusch--Stanley formula~\cite{billey1993some}
\begin{equation}
\mathfrak{S}_w = \sum_{ (a_1,\dots,a_p) \in \Red(w)} \sum_{i_1 \leq \dots \leq i_p} x_{i_1} \dots x_{i_p}.
\label{eq:bjs}	
\end{equation}
Here $i_j \leq a_j$ for all $j$ and $i_k = i_{k+1}$ only if $a_k < a_{k+1}$.
Biwords $\binom{i_1 \dots i_p}{a_1 \dots a_p}$ satisfying these conditions are called \emph{reduced compatible sequences}.
Let $\RC(w)$ denote the set of reduced compatible sequences for $w$.

No combinatorial rule for multiplying arbitrary Schubert polynomials is known, but the product is determined by Monk's formula.
\begin{theorem}[Monk's formula]
	\label{t:monk}
	For $w \in S_\infty$ and $k \in \mathbb{P}$,
	\[
	\fkS_{s_k} \cdot \fkS_v = \sum_{v \xrightarrow{k} w} \fkS_w.
	\]
\end{theorem}
Since $\fkS_k = x_1 + \dots + x_k$, one could give a bijective proof of Monk's formula by constructing a bijection
\[
[k] \times \RC(v) \xrightarrow{} \bigcup_{v \xrightarrow{k} w} \RC(w)
\]
so that the image of $(j,\binom{i_1 \dots i_p}{a_1 \dots a_p})$ has weight $x_j \cdot x_{i_1} \dots x_{i_p}$.
One such bijection, which we denote $\BB$, appears in~\cite{bergeron1993rc}.
Another bijection appears in~\cite{billey2019bijective} building on the Little bijection~\cite{little2003combinatorial}, and is attributed to Buch in~\cite{knutsonschubert}.

\subsection{Dual Schubert polynomials} We introduce some additional properties of dual Schubert polynomials from~\cite{postnikov2009chains}.
Let $X_i = x_i - x_{i+1}$.
Then for $u \xrightarrow{} v$ with $v = u (a,b)$, we have 
\[
\mu(u \xrightarrow{}v) := x_a - x_b = X_a + \dots + X_{b-1}.
\]
Furthermore, $u \xrightarrow{d} v$ precisely when $a \leq d < b$.
Letting the label $d$ identify the term $X_d$ in $\mu(u,v)$, we have the apparently novel reformulation of Equation~\eqref{eq:dual-schub}:
\begin{proposition}
	\label{p:dual-schub}
	For $u, v \in S_\infty$,
	\[
	\fk D_{u,v} = \frac{1}{p!}\sum_{C: u = v_0 \xrightarrow{d_1} v_1 \xrightarrow{d_2} \dots \xrightarrow{d_p} v_p = w} X_{d_1} \dots X_{d_p}.
	\]
\end{proposition}
Recall Schubert polynomials and dual Schubert polynomials are dual with respect to the inner product $\langle \cdot ,\cdot  \rangle$ defined in Equation~\eqref{eq:inner_product}.
Moreover, in~\cite[\S6]{postnikov2009chains} the authors show skewing a dual Schubert polynomial is adjoint to multiplying by a Schubert polynomial:
\[
\langle \fkS_u \cdot \fkS_v , \fk D_w\rangle = \langle \fkS_v,\fk D_{u,w} \rangle = c_{uv}^w,
\]
where the $c_{uv}^w$'s are Littlewood--Richardson coefficients.
As an immediate consequence, we have
\begin{equation}
\label{eq:skew}
\fk D_{u,w} = \sum_{v} c_{uv}^w \fk D_{1,v}.	
\end{equation}
Combining Equation~\eqref{eq:skew} with Proposition~\ref{p:dual-schub}, we have:
\begin{corollary}
\label{c:chain-symmetry}
	For $u,w \in S_\infty$ with $u \xrightarrow{\textbf{d}} w$, we have
	\[
	|\mathcal{C}_\textbf{d}(u,w)| = \sum_v c_{uv}^w |\mathcal{C}_\textbf{d}(1,v)|, \  \mbox{which refines} \ |\mathcal{C}(u,w)| = \sum_v c_{uv}^w |\mathcal{C}(1,v)|.
	\]
\end{corollary}

\section{$k$--chains and the cauchy identity}

We are now prepared to give our novel proof that the definitions of dual Schubert polynomials in terms of chains and $\langle \cdot, \cdot \rangle$ are equivalent, that is that $\fk{D}_w = \fk{D}_{1,w}$.

\begin{proof}[Proof of Theorem~\ref{t:dual-def}]
We first explain how to interpret the left hand side of Equation~\eqref{eq:cauchy} combinatorially.
For $Y_i = y_i - y_{i+1}$, we have $y_i = Y_i + Y_{i+1} + \dots$ so 
\begin{align*}
e^{x_1y_1 + x_2y_2 + \dots} &= \sum_{n} \frac{1}{n!}(x_1y_1 + x_2y_2 + \dots)^n\\
&= \sum_{n} \frac{1}{n!} \left(x_1 (Y_1 + Y_2 + \dots) + x_2 (Y_2 + Y_3 + \dots) + \dots\right)^n\\
&= \sum_{n}\frac{1}{n!} \sum_{\binom{i_1 \dots i_n}{d_1 \dots d_n},\  i_j \leq d_j} x_{i_1}Y_{d_1} \dots x_{i_n}Y_{d_n}.
\end{align*}
Following~\cite[Def 3.2]{huang2022bumpless}, we call the words appearing in the final sum \emph{bounded biwords}.
Therefore, to prove Equation~\eqref{eq:cauchy} bijectively, one must demonstrate a bijection from bounded biwords to pairs $(R,C)$ where $R$ is a reduced compatible sequence and $C$ a labeled chain for the same permutation.
This follows by repeated application of the map $\BB$, or any other bijective proof of Monk's formula.
\end{proof}

\begin{remark}
\label{rem:bij}
As observed in~\cite{huang2022bumpless}, one can interpret the repeated application of $\BB$ or any other bijective proof of Monk's rule as an insertion-like algorithm, with the compatible sequence as the insertion object and the Bruhat chain as the recording object.
The proof of Theorem~\ref{t:dual-def} shows these insertions can be viewed as bijective proofs of the Cauchy identity Equation~\eqref{eq:cauchy}.
\end{remark}

While bounded biwords do not have any restriction on order, one can insist they satisfy monotonicity conditions via a simple scaling factor.
Iterating Monk's formula with $\textbf{d} = (d_1,\dots,d_p)$ gives
\begin{equation}
	\label{eq:chain-symmetry}
	(\fk S_{s_{d_p}} \dots \fk S_{s_{d_1}}) \fk S_u = \sum_{C:u \xrightarrow{\textbf{d}} w} \fk S_w.
\end{equation}
Since the left hand side is symmetric, we see:
\begin{corollary}
Let $u,w \in S_\infty$ with $u \xrightarrow{\textbf{d}} w$.
If $\textbf{d}'$ is a permutation of $\textbf{d}$, then 
\[
|\mathcal{C}_\textbf{d}(u,w)| = |\mathcal{C}_{\textbf{d}'}(u,w)|.
\]

\end{corollary}

A bijective proof of this fact can be deduced using Lenart's growth diagrams~\cite{lenart2010growth}.
 
Fix a total order $\prec$ on $\P$.
A labeled chain $u \xrightarrow{\textbf{d}} w$ is \emph{$\prec$--increasing} if the labels are weakly increasing with respect to $\prec$.
Let $\mathcal{C}_\prec(u,w)$ be the set of labeled chains from $u$ to $w$ that are $\prec$--increasing.
For $\textbf{d} = (d_1,\dots d_p)$, define $\alpha(\textbf{d}) = (\alpha_1, \alpha_2, \dots)$  with $\alpha_i = \#\{j: d_j = i\}$.
Note there are $p!/\alpha(\textbf{d})!$ permutations of $\textbf{d}$ where $\alpha(\textbf{d})!$ is the product $\alpha_1! \cdot \alpha_2! \cdot \dots$.
Define the linear transformation $\overline{\;\cdot\;}$ by $\overline{x^\alpha} = \frac{p!}{\alpha!}x^\alpha$.
\begin{corollary}
\label{c:increasing}
For $u,w \in S_\infty$ with $u \leq w$ and $\prec$ a total order on $\P$, we have
\[\overline{\fk D_{u,w}} = \frac{1}{\alpha!}\sum_{C:u \xrightarrow{\textbf{d}} w\; \in \; \mathcal{C}_\prec(u,w)} X_{d_1} \dots X_{d_p} = \sum_{v} c_{uv}^w \overline{\fk D_v}.
\]	
\end{corollary}

\begin{proof}
The first equality follows from Proposition~\ref{p:dual-schub} and the definition of $\overline{\;\cdot\;}$, while the latter follows from Equation~\eqref{eq:dual-schub}.	
\end{proof}

The total order $\prec$ allows us to prescribe any order on the values occurring in $\textbf{d}$.
For example, if $\prec$ is $<$, we are requiring $d_1 \leq d_2 \leq \dots \leq d_p$, while if $\prec$ has $2 \prec 1 \prec 3 \prec 4 \prec \dots$, we are requiring $\textbf{d}$ first have $2$'s, then $1$'s and so on.

\begin{remark}
	In~\cite{huang2022bumpless}, Huang and Pylyavskyy interpret certain Littlewood--Richardson coefficients as enumerating certain descending chains, or $>$--increasing in our notation.
Corollary~\ref{c:increasing} shows this is a general phenomenon.
One interpretation of Corollary~\ref{c:increasing} is that $\prec$--increasing chains ``know'' Littlewood--Richardson coefficients.
As observed in~\cite{huang2023knuth}, restricted sets of chains can prove easier to analyze.

More generally, one can use Proposition~\ref{p:dual-schub} to show certain chains ``know'' certain parts of Schubert calculus.
For example,  Corollary~\ref{c:chain-symmetry} shows
\[
|\mathcal{C}_{(k,\dots,k)}(u,w)| = \sum_{v} c_{uv}^w |\mathcal{C}_{(k,\dots,k)}(1,v)|.
\]
Since $v$ for which $\mathcal{C}_{(k,\dots,k)}(1,v)$ is non-empty is necessarily $k$--Grassmannian, and $\fk{S}_v$ is a Schur polynomial $s_\lambda(x_1,\dots,x_k)$, we recover Bergeron and Sottile's observation that $k$--chains ``know'' how to compute the product $\fkS_w \cdot s_\lambda(x_1,\dots,x_k)$~\cite{bergeron1998schubert}.
\end{remark}

\section*{Acknowledgements}

The author was partially supported by NSF Grant DMS-2054423 and thanks Shiliang Gao, Daoji Huang, Oliver Pechenik, Pasha Pylyavskyy and Tianyi Yu for helpful conversations and suggestions.

\bibliographystyle{plain} 
\bibliography{references}

\begin{thebibliography}{10}

\bibitem{bergeron1993rc}
Nantel Bergeron and Sara Billey.
\newblock {RC}-graphs and {S}chubert polynomials.
\newblock {\em Experimental Mathematics}, 2(4):257--269, 1993.

\bibitem{bergeron1998schubert}
Nantel Bergeron and Frank Sottile.
\newblock Schubert polynomials, the {B}ruhat order, and the geometry of flag
  manifolds.
\newblock {\em Duke Mathematical Journal}, 95:373--423, 1998.

\bibitem{bernstein1973schubert}
I.~N. Bernstein, I.~M. Gel'fand, and S.~I. Gel'fand.
\newblock Schubert cells and cohomology of the spaces {G}/{P}.
\newblock {\em Russ. Math. Surv.}, 28(3):1--26, 1973.

\bibitem{billey2019bijective}
Sara~C Billey, Alexander~E Holroyd, and Benjamin~J Young.
\newblock A bijective proof of {M}acdonald's reduced word formula.
\newblock {\em Algebraic Combinatorics}, 2(2):217--248, 2019.

\bibitem{billey1993some}
Sara~C Billey, William Jockusch, and Richard~P Stanley.
\newblock Some combinatorial properties of {S}chubert polynomials.
\newblock {\em Journal of Algebraic Combinatorics}, 2(4):345--374, 1993.

\bibitem{huang2022bumpless}
Daoji Huang and Pavlo Pylyavskyy.
\newblock Bumpless pipe dream {RSK}, growth diagrams, and {S}chubert structure
  constants.
\newblock {\em arXiv preprint arXiv:2206.14351}, 2022.

\bibitem{huang2023knuth}
Daoji Huang and Pavlo Pylyavskyy.
\newblock Knuth moves for schubert polynomials.
\newblock {\em arXiv preprint arXiv:2304.06889}, 2023.

\bibitem{knutsonschubert}
Allen Knutson.
\newblock Schubert polynomials and symmetric functions notes for the {L}isbon
  {C}ombinatorics {S}ummer {S}chool 2012.

\bibitem{lascoux1982structure}
Alain Lascoux and Marcel-Paul Sch{\"u}tzenberger.
\newblock Structure de {H}opf de l'anneau de cohomologie et de l'anneau de
  {G}rothendieck d’une vari{\'e}t{\'e} de drapeaux.
\newblock {\em CR Acad. Sci. Paris S{\'e}r. I Math}, 295(11):629--633, 1982.

\bibitem{lenart2010growth}
Cristian Lenart.
\newblock Growth diagrams for the {S}chubert multiplication.
\newblock {\em Journal of Combinatorial Theory, Series A}, 117(7):842--856,
  2010.

\bibitem{little2003combinatorial}
David~P Little.
\newblock Combinatorial aspects of the {L}ascoux--{S}ch{\"u}tzenberger tree.
\newblock {\em Advances in Mathematics}, 174(2):236--253, 2003.

\bibitem{manivel2001symmetric}
Laurent Manivel.
\newblock {\em Symmetric functions, {S}chubert polynomials and degeneracy
  loci}, volume~3.
\newblock American Mathematical Soc., 2001.

\bibitem{monk1959geometry}
David Monk.
\newblock The geometry of flag manifolds.
\newblock {\em Proceedings of the London Mathematical Society}, 3(2):253--286,
  1959.

\bibitem{postnikov2009chains}
Alexander Postnikov and Richard~P Stanley.
\newblock Chains in the {B}ruhat order.
\newblock {\em Journal of Algebraic Combinatorics}, 29(2):133--174, 2009.

\end{thebibliography}

\end{document}